\documentclass[11pt]{article}

\usepackage{amsmath,epsfig,amssymb,amsthm}
\usepackage{graphicx,subfigure}
\usepackage{charter}
\usepackage{microtype}

\def\w{\omega}
\def\x{\boldsymbol{x}}

\def\Hil{\mathcal H}
\def\R{\mathbf{R}}
\def\eps{\varepsilon}
\newcommand{\abs}[1]{\left\lvert #1\right\rvert}

\newtheorem{theorem}{Theorem}[section]

\newtheorem{proposition}[theorem]{Proposition}

\begin{document}

\title{On the Hilbert transform of wavelets}

\author{Kunal~Narayan~Chaudhury\thanks{Correspondence: Kunal~N.~Chaudhury (kchaudhu@princeton.edu). The second author is with the Biomedical Imaging Group, \'Ecole Polytechnique Federale de Lausanne, 
Switzerland; Fax: +41 216933701. This work was supported in part by the Swiss National Science Foundation under grant 200020-109415.} \ and Michael~Unser}

\maketitle

\begin{abstract}

	A wavelet is a localized function having a prescribed number of vanishing moments. In this correspondence, we provide precise arguments as to why the Hilbert transform of a wavelet is again a wavelet. In particular, we provide sharp estimates of the localization, vanishing moments, and smoothness of the transformed wavelet.  We work in the general setting of non-compactly supported wavelets. Our main result is that, in the presence of some minimal smoothness and decay, the Hilbert transform of a wavelet is again as smooth and oscillating as the original wavelet, whereas its localization is controlled by the number of vanishing moments of the original wavelet. We motivate our results using concrete examples.
	
\end{abstract}


\section{Introduction}

	It is known that the poor translation-invariance of standard wavelet bases can be improved by considering a pair of wavelet bases, whose mother wavelets are related through the Hilbert transform \cite{kingsbury1,kingsbury2,CTDWT,Pasquet}. The advantages of using Hilbert wavelet pairs for signal analysis had also been recognized by other authors \cite{Flandrin,Daubechis}. More recently, it was shown in \cite{Chaudhury2009} how a Gabor-like wavelet transform could be realized using such Hilbert pairs.

	The fundamental reasons why the Hilbert transform can be seamlessly integrated into the multiresolution framework of wavelets are its scale and translation invariances, and its energy-preserving (unitary) nature \cite{Chaudhury2009}. These properties are at once obvious from the Fourier-domain definition of the transform. We recall that the Hilbert transform $\Hil f(x)$ of a sufficiently well-behaved function $f(x)$ is specified by
\begin{equation}
\label{freq-def}
\widehat{\Hil f}(\w)=
\begin{cases} -j \hat f(\w)     & \text{ for } \w >0 \\
+j \hat f(\w)    & \text{ for } \w <0 \\
0    & \text{ at } \w=0. \\
\end{cases}
\end{equation}
On one hand, the unitary nature ensures that the Hilbert transform of a (wavelet) basis of $\mathrm{L}^2(\R)$ is again a basis of $\mathrm{L}^2(\R)$. On the other hand, the invariances of scale and translation together provides coherence---the Hilbert transform of a wavelet basis generated from the mother wavelet $\psi(x)$ is simply the wavelet basis generated from the mother wavelet $\Hil \psi(x)$.

	The flip side, however, is that the transform is incompatible with scaling functions (low-pass functions in general), the building blocks of multiresolution analyses. As shown in Figure \ref{HilbertLowpass}, the transform  ``breaks-up'' scaling functions, resulting in the loss of their crucial approximation property. Moreover, the transformed function exhibits a slow decay. Starting from a given multiresolution with associated wavelet basis $(\psi_n)_{n \in \mathbf Z}$, this presents conceptual difficulties in realizing a dual multiresolution with basis $(\Hil \psi_n)_{n \in \mathbf Z}$. It was shown in \cite{Chaudhury2009} that this pathology can, however, be overcome by a careful design of the dual multiresolution in which the Hilbert transform is applied only on the wavelet, and never explicitly on the scaling function. 
	
	The above-mentioned pathologies can be explained by considering the space-domain definition of the transform, which is slightly more involved mathematically (see, e.g., \cite{Stein,Javier}):
\begin{equation}
\label{def}
\Hil f(x)=\frac{1}{\pi} \lim \limits_{\eps \rightarrow 0} \int_{|t| > \eps} \!\!\! f(x-t) \ \frac{dt}{t}.
\end{equation}
Disregarding the technicalities involving the use of truncations and limits, $\Hil f(x)$ is thus essentially given by the convolution of $f(x)$ with the kernel $1/\pi x$ (cf. Figure \ref{SingularKernel}). It is now readily seen that the above-mentioned observations follow as a consequence of the ``oscillating'' form of the kernel and its slow decay at the tails. We will conveniently switch between definitions \eqref{freq-def} and \eqref{def} in the sequel.

\begin{figure*} 
\centering 
\includegraphics[width=1.0\linewidth]{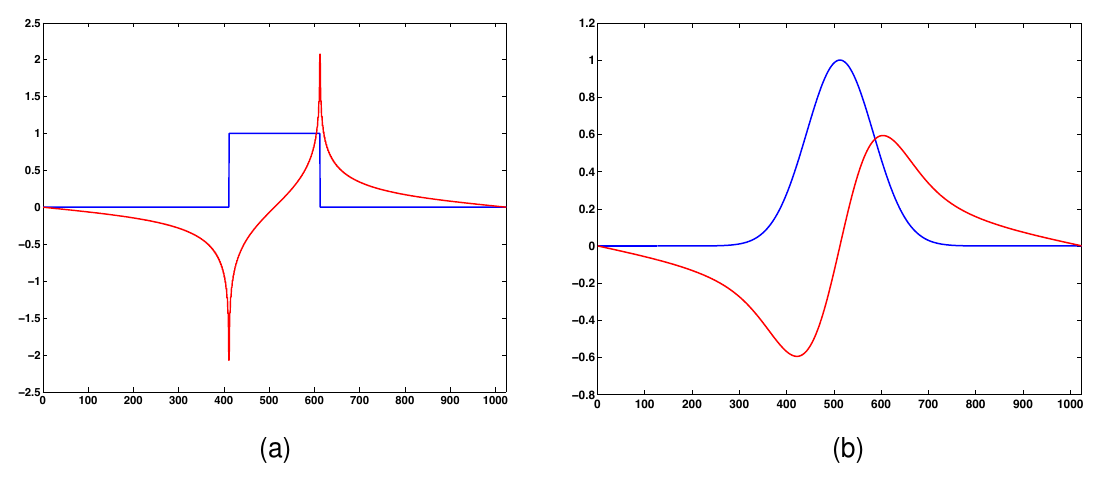}
\caption{Scaling functions and their Hilbert transforms: $(a)$ The discontinuous Haar scaling function (BLUE) and its transform (RED), $(b)$ The smooth cubic B-spline (BLUE) and its transform (RED). In either case, the  transformed function is ``broken-up'' and, as a consequence, loses its approximation property. In particular, the transform no longer exhibits the partition-of-unity property, which is characteristic of scaling functions. Also, note the slow decay of the transform, particular for the smooth spline function. In fact, both the transforms decay as $1/|x|$---the smoothness of the original function has no effect on the decay of the transform.} 
\label{HilbertLowpass} 
\end{figure*}

	Our main observation is that the Hilbert transform goes well with oscillatory patterns, and wavelets in particular. The archetypal relation in this regard is its action on pure sinusoids,
\begin{equation*}
\Hil [\cos(\w_0 x)]=\sin(\w_0 x).
\end{equation*}
Thus, the transform tends to preserve oscillations. The nature of the interaction with localized oscillations is suggested by the relation
\begin{equation}
\label{modulation}
\Hil [ \varphi(x) \cos(\w_0x)] = \varphi(x) \sin(\w_0 x)
\end{equation}
which holds if the localization window $\varphi(x)$ is bandlimited to $(-\w_0, \w_0)$ \cite{Bedrosian}. This is an immediate consequence of definition \eqref{freq-def}. The crucial observation, however, is that the transformed signal is again smooth (in fact, infinitely differentiable) and oscillatory, and importantly, has the same localization  as the input signal. It is known that a particular family of spline wavelets, namely, the B-spline wavelets \cite{convergence}, converge to a function of the form  $\varphi(x) \cos(\w_0x+\phi)$ with the increase in the order of the spline. In particular, it was shown in \cite{Chaudhury2009} that the Hilbert transform has comparable  localization, smoothness, and vanishing moments for sufficiently large orders (cf. Figure \ref{HilbertWavelets}). It was also shown that the transformed wavelet in fact approaches $\varphi(x) \sin(\w_0 x+\phi)$ as the order increases, which is consistent with \eqref{modulation}. Since, more generally, wavelets with sufficient smoothness and vanishing moments can be made to closely approximate the form in \eqref{modulation}, we could in fact arrive at a similar conclusion for a larger class of wavelets. 

Using these instances as guidelines, we attempt to answer the following basic questions in the sequel:
\begin{itemize}
\item When is the Hilbert transform of a wavelet well-defined? In particular, how much smoothness and decay is required?
\item Why does the Hilbert transform of a wavelet exhibit better decay than the corresponding scaling function? How does one really get past the $1/|x|$ decay?
\item How good is the localization of the transformed wavelet, how smooth is it, and how many vanishing moments does it have? 
\end{itemize}

\section{Notations}

	The Fourier transform of $f(\x)$ is defined by $\hat{f}(\w)=\int_{-\infty}^{\infty} f(x) \exp{(-j \w x)} \ dx$. We omit the domain of integration when this is obvious from the context.  We define $\lVert f \rVert_1=\int_{-\infty}^{\infty} |f(x)| \ dx$, and $\lVert f \rVert_{\infty}=\sup \{ |f(x)|: x \in \R\}$. The notation $T_x f(t)$ denotes the function $T_x f(t)=f(x-t)$. We write $f(x)=O(g(x)), x \in A,$ to signify that $|f(x)| \leq C g(x)$ for all $x \in A$, where $C$ is an absolute constant. We denote the first derivative of $f(x)$ by $f'(x)$; in general, we denote the $k$-th derivative by $f^{(k)}(x)$. We say that $f(x)$ is $n$-times continuously differentiable if all its derivatives up to order $n$ exists and are continuous.

\begin{figure*} 
\centering 
\includegraphics[width=1.0\linewidth]{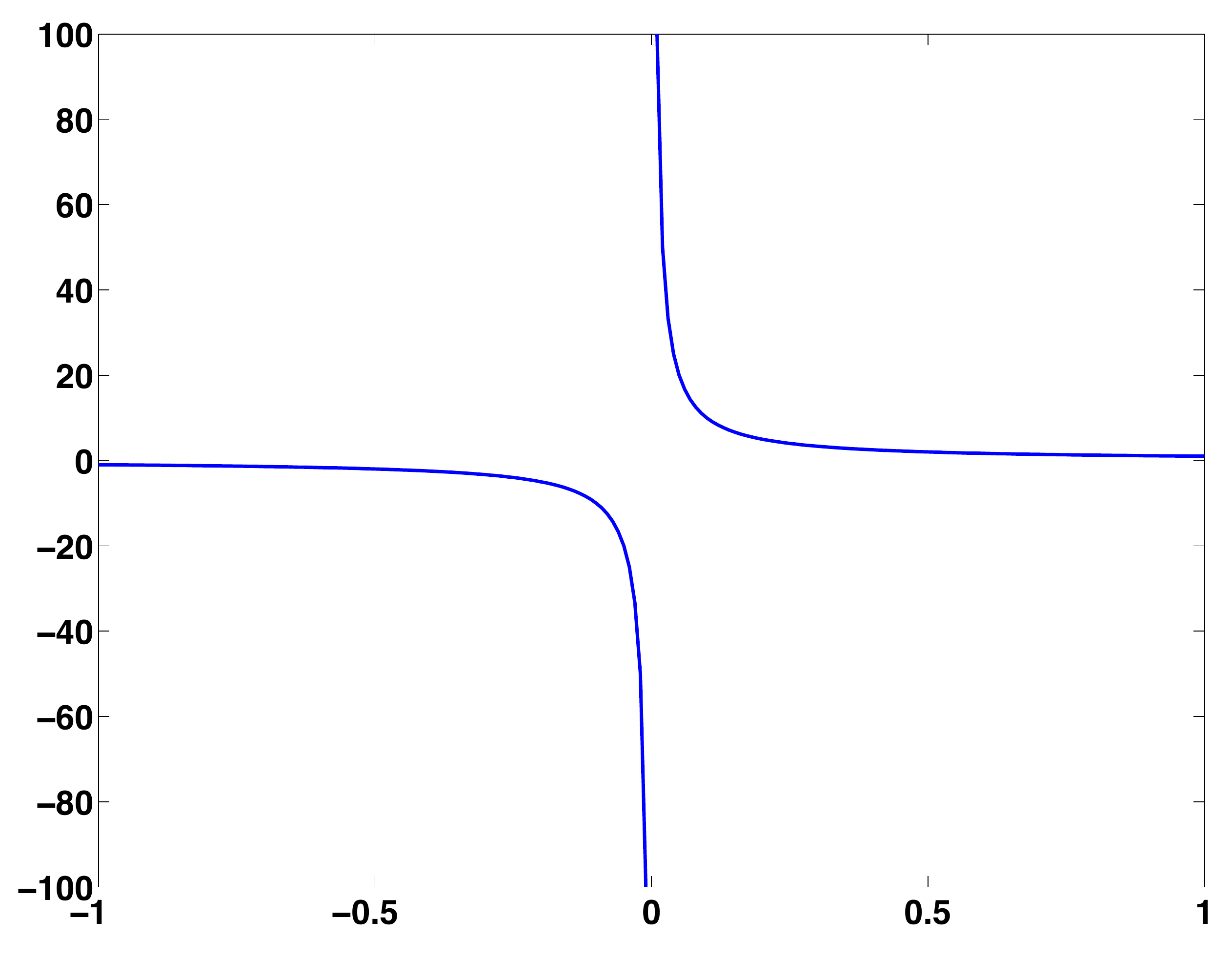}
\caption{The convolution kernel of the Hillbert transform, $1/\pi x$. It has a singularity at the origin and its tails decay slowly. The former pathology can be overcome provided that the signal on which the transform is applied is sufficiently smooth, while the slow decay can be overcome if the signal is of compact support, or at least, of sufficient decay.} 
\label{SingularKernel} 
\end{figure*}

\section{Main results}
	
	The kernel $1/\pi x$ fails to be absolutely integrable owing to its slow decay and, more importantly, its singularity at the origin. The limiting argument in \eqref{def} avoids the singularity by truncating the kernel around the origin in a systematic fashion. The slow decay of the kernel, on the other hand, can be dealt with by simply restricting the domain of \eqref{def} to functions with sufficient decay. 
		
	As noted in the introduction, the Hilbert transform goes well only with smooth functions. This can be readily appreciated by looking at the transform of the discontinuous Haar wavelet in Figure \ref{HilbertWavelets}. In this case, the transform ``blows-up'' in the vicinity of the discontinuities, and is, in fact, not even well-defined at the points of discontinuity. The following result, which relies on some classical methods of harmonic analysis, explains how this problem can be fixed. 
	
	For convenience, we introduce the mixed norm $\lVert  f \rVert_{1,\infty}=\lVert f \lVert_1+ \lVert f' \lVert_{\infty}$ which measures both the local smoothness and the global size of $f(x)$.

\begin{theorem}[The classical result]
\label{decay}
Let $f(x)$ be a differentiable function such that both $ \|f\|_{1,\infty}$ and $ \| x f(\cdot)\|_{1,\infty}$  are finite. Then $\Hil f(x)$ is well-defined, and 
\begin{equation}
\label{bound}
| \Hil f(x) | \leq \frac{C}{1+|x|} \left(\|f\|_{1,\infty} + \| x f(\cdot)\|_{1,\infty} \right )=O(|x|^{-1}).
\end{equation}
\end{theorem}

	In particular, this holds true if $f(x)$ is continuously differentiable and is of compact support.

\begin{proof}
Consider the basic quantity
\begin{equation}
\label{functional}
W(f) = \frac{1}{\pi} \lim_{\eps \rightarrow 0} \int_{|t| > \eps} f(t) \ \frac{dt}{t}.
\end{equation}
Note that $\Hil f(x)= W (T_x f)$. To begin with, we at least need to guarantee that $W(f)$ is well-defined. Note that the integrand in \eqref{functional} is the product of the bounded function $1/t$ (on $\abs{t} > \eps$) and the integrable function $f(t)$. Therefore, the integral is absolutely convergent for all $\eps >0$. All we need to show is that the integral remains convergent as $\eps \longrightarrow 0$.
To this end, we split the integral in \eqref{functional}, and use the anti-symmetry of $1/t$ to write
\begin{align*}
W(f)  & = \frac{1}{\pi} \lim_{\eps \rightarrow 0} \int_{\eps < |t| < 1} f(t) \ \frac{dt}{t} + \frac{1}{\pi} \int_{|t| \geq 1} f(t) \ \frac{dt}{t} \\
& = \frac{1}{\pi} \lim_{\eps \rightarrow 0} \int K(\eps,t)     \frac{f(t)-f(0)}{t} \ dt +\frac{1}{\pi} \int_{|t| \geq 1} f(t) \ \frac{dt}{t}.
\end{align*}
where $K(\eps,t)=1$ when $\eps < |t| <1$ and zero otherwise. Clearly, the second integral is convergent. As for the first, note that since $f'(x)$ is bounded, by the mean-value theorem, $K(\eps,t)   |(f(t)-f(0)/t| \leq  \lVert f' \lVert_{\infty}$ for $t$ and for all $\eps >0$. Therefore, by the dominated convergence theorem, 
\begin{equation*}
\lim_{\eps \rightarrow 0} \int K(\eps,t)    \Big|\frac{f(t)-f(0)}{t} \Big| \ dt \leq  2 \lVert f' \lVert_{\infty}.
\end{equation*}

	In particular, we conclude that $W(f)$ is well-defined, and 
\begin{align}
\label{E1}
|W(f)| \leq \frac{1}{\pi}  \left( 2 \lVert f' \lVert_{\infty} +  \lVert f \lVert_1 \right).
\end{align}
Since $T_x f(t)$ has the same decay and smoothness as $f(t)$, it is now immediate that $\Hil f(x)$ is well-defined (pointwise), and that
\begin{equation}
\label{bound1}
| \Hil f(x) | \leq \frac{1}{\pi}  \|T_x f\|_{1,\infty}= \frac{1}{\pi} \left( 2 \lVert f' \lVert_{\infty} +  \lVert f \lVert_1 \right).
\end{equation}

Next, we note that 
\begin{align*}
x \Hil f(x) & = \frac{1}{\pi} \lim_{\eps \rightarrow 0} \int_{|t| > \eps} (x-t) f(x-t) \ \frac{dt}{t} + \frac{1}{\pi} \int f(t) \ dt \\
& = \Hil g(x) +\frac{1}{\pi} \int f(t) \ dt 
\end{align*}
where $g(x)=xf(x)$. Since $\|g\|_{1,\infty}$ is finite, $\Hil g(x)$ is well-defined, and 
\begin{equation*}
| \Hil g(x) | \leq \frac{1}{\pi}   \left( 2 \lVert g' \lVert_{\infty} +  \lVert g \lVert_1 \right).
\end{equation*}
Therefore,
\begin{equation}
\label{bound2}
| x \Hil f(x) |   \leq  \frac{1}{\pi} \left( 2 \lVert g' \lVert_{\infty} +  \lVert g \lVert_1+ \|  f \|_1 \right).
\end{equation}
Combining \eqref{bound1} and \eqref{bound2}, we obtain \eqref{bound}.
\end{proof}

We note that the main conclusions of the theorem are well-known results in harmonic analysis; e.g., see \cite{Stein,Javier}. Moreover, the assumptions under which we reproduce these results in Theorem \ref{decay} are on the conservative side. In fact, as can already be seen from our derivation, the transform remains well-defined if we replace the constraint $\lVert f' \lVert_{\infty} < \infty$ by the weaker hypothesis of Lipschitz continuity, that is, if $|f(x)- f(y)| \leq C |x-y|$ for some absolute constant $C$. Our goal here was to introduce some mathematical tools which we eventually use to prove our main result.

\subsection{Vanishing moments and decay}

The derivation of the Theorem \ref{decay} exposes the unfortunate fact that the poor $1/|x|$ decay cannot be improved even if $f(x)$ is required to be more smooth (cf. transform of the cubic spline in Figure \ref{HilbertLowpass}), or have a better decay. However, it does suggest the following: If $\Hil g(x)$ goes to zero as $|x|$ goes to infinity (which is the case if $g(x)$ is sufficiently nice), then
\begin{equation*}
\lim_{|x| \longrightarrow \infty} x \Hil f(x) = \frac{1}{\pi} \int f(x) \ dx.
\end{equation*}
In particular, if $f(x)$ has zero mean, then $x \Hil f(x)$ goes to zero at infinity. Therefore, the decay of $\Hil f(x)$ must be better than $1/|x|$ in this case. This alludes to the connection between the zero-mean condition and the improvement in decay. To make this more precise, we consider the example of the Haar wavelet
\begin{equation*}
\psi(x)=
\begin{cases} +1 & \text{ for } -1 \leq x < 0 \\
-1 & \text{ for }  0 \leq x < 1.
\end{cases}
\end{equation*}
Let $\abs{x} >2$. Since $\psi(x)$ has zero mean, we can write
\begin{align*}
 \Hil \psi(x)  & =\frac{1}{\pi}  \lim_{\eps \rightarrow 0} \int_{|t-x| > \eps} \frac{\psi(t)}{x-t} \ dt  \\
 & =\frac{1}{\pi}  \lim_{\eps \rightarrow 0} \int_{|t-x| > \eps} \psi(t) \left(\frac{1}{x-t}-\frac{1}{x}\right) \ dt  \\
 & = \frac{1}{\pi} \lim_{\eps \rightarrow 0} \int_{|t-x| > \eps} \frac{t \psi(t) }{x(x-t)} \ dt.
\end{align*}
Now $\abs{x-t} \geq \abs{x}$/2 for $\abs{x}>2$, and $t \in [-1,1]$. Hence,
\begin{align*}
| \Hil \psi(x)|  \leq \frac{2}{\pi \abs{x}^2} \int_{-1}^1  |t  \psi(t) | \ dt \leq \frac{1}{\pi \abs{x}^2}.
\end{align*}
Thus, while the Hilbert transform of the Haar scaling function decays only as $1/\abs{x}$, the transform of the Haar wavelet has a better decay of $1/\abs{x}^2$. This is clearly seen by comparing the plots in Figures \ref{HilbertLowpass} and \ref{HilbertWavelets}. 

	We can now generalize the above observation by requiring that, for some $n \geq 1$,
\begin{equation*}
\int x^k \psi(x) \ dx=0 \qquad (0 \leq k < n).
\end{equation*}
This \textit{vanishing moment} property is in fact characteristic of wavelets, which are often parametrized by the number $n$  \cite{WTSP}. The following result explains how higher vanishing moments can contribute to the increase in the decay of the Hilbert transform. The main idea is that the kernel of the Hilbert transform effectively behaves as $1/\pi x^{n+1}$ in the presence of $n$ vanishing moments.

We use the augmented decay to compute the number of vanishing moments of the transformed wavelet.   

\begin{figure*}
\centering
\includegraphics[width=1.0\linewidth]{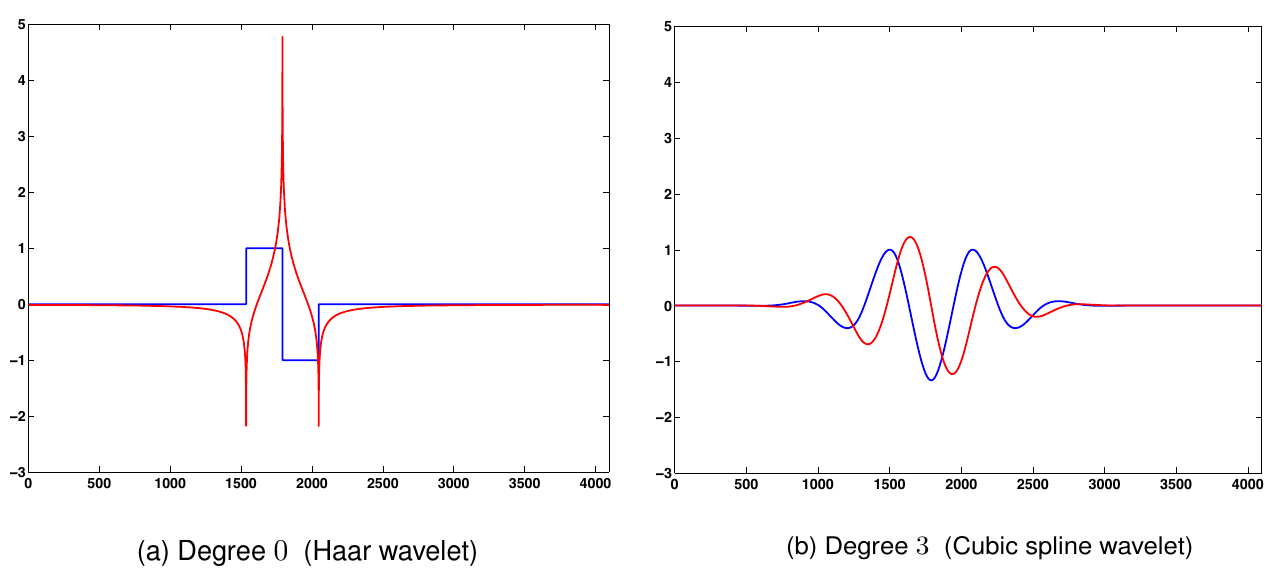} 
\caption{B-spline wavelets (shown in BLUE) and their Hilbert transforms (shown in RED). The wavelets are ordered (left to right) by increasing smoothness and vanishing moments; both are compactly supported. Notice how the decay of the Hilbert transform increases with the increase in vanishing moments---the transform of the cubic spline wavelet appears to have an almost identical localization. Moreover, it is as smooth as the original wavelet. It is shown in the text that, in the presence of some minimal smoothness, the Hilbert transform is as smooth and oscillating as the spline wavelet.}
\label{HilbertWavelets}
\end{figure*}

\begin{theorem}[Decay and vanishing moments]
\label{better-decay}
Let $\psi(x)$ be a differentiable wavelet having $n$ vanishing moments.  Also, assume that $ \|\psi \|_{1,\infty}, \| x^{n+1} \psi(\cdot)\|_{1,\infty}$, and $ \| x^n \psi(\cdot)\|_1$ are finite. Then $\Hil \psi(x)$ is well-defined, and 
\begin{equation}
\label{bound4}
| \Hil \psi(x) | \leq \frac{C}{1+|x|^{n+1}} \Big( \|\psi\|_{1,\infty} + \|x^{n+1} \psi(\cdot)\|_{1,\infty} + \| x^n \psi(\cdot)\|_1 \Big )=O(|x|^{-n-1}).
\end{equation}
Moreover, $\Hil \psi(x)$ has $n$ vanishing moments.
\end{theorem}

	Before proceeding to the proof, we make some comments. Note that, under the assumptions on the vanishing moments, \eqref{bound4} holds true for compactly supported wavelets provided it is continuously differentiable. This in fact is the case for the cubic spline wavelet shown in Figure \ref{HilbertWavelets}. More generally, \eqref{bound4} holds if $\psi(x)$ is continuously differentiable, has $n$ vanishing moments, and satisfies the mild decay conditions  
\begin{equation*}
\psi(x)=O(1/\abs{x}^{n+3+\eps}), \quad  \psi'(x)=O(1/\abs{x}^{n+2+\eps'}) \qquad (\abs{x} \longrightarrow \infty)
\end{equation*}
where $\eps$ and $\eps'$ are arbitrarily small positive numbers. The significance of the above result is that by requiring $\psi(x)$ to have a large number of vanishing moments, we can effectively make $\Hil \psi(x)$ as localized as $\psi(x)$. This had been observed qualitatively early on in connection with the wavelet localization of the Radon transform \cite{Olson}. 

	Now, we show that \eqref{bound4} is sharp, by considering the special case of B-spline wavelets. It is known that if $\psi(x)$ is a B-spline wavelet of degree $n-1$, then $\Hil \psi(x)$ is again a (fractional) B-spline wavelet of the same degree, and hence has the same decay of $1/|x|^{n+1}$ \cite{Chaudhury2009}, \cite{FractionalSplines2000}. This exactly what is predicted by \eqref{bound4}, since $\psi(x)$ is known to have $n$ vanishing moments.

\begin{proof}[Proof of Theorem \ref{better-decay}]
It follows from Theorem \ref{decay} that $\Hil \psi(x)$ is well-defined, and that
\begin{equation}
\label{bound5}
| \Hil \psi(x) | \leq  \frac{1}{\pi} \left( 2 \lVert \psi' \lVert_{\infty} +  \lVert \psi \lVert_1 \right).
\end{equation}
As for the decay, fix any $x$ away from zero, and let 
 \begin{equation*}
P(t)=\frac{1}{x} + \frac{t}{x^2} + \cdots + \frac{t^{n-1}}{x^n}.
\end{equation*}
It is clear that
\begin{equation*}
\int P(t) \psi(t) \ dt=0.
\end{equation*}
Using this, we can write
\begin{align*}
\Hil \psi(x)& = \frac{1}{\pi} \lim_{\eps \rightarrow 0} \int_{|t-x| > \eps} \frac{\psi(t)}{x-t} \ dt  \\
& =\frac{1}{\pi} \lim_{\eps \rightarrow 0} \int_{|t-x| > \eps} \psi(t) \Big(\frac{1}{x-t} -P(t) \Big) \ dt.
\end{align*}
A simple computation shows that
\begin{equation*}
\frac{1}{x-t} -P(t)= \frac{1}{x^{n+1}}\Big(t^n+\frac{t^{n+1}}{x-t}\Big),
\end{equation*}
so that
\begin{align*}
\Hil \psi(x) &= \frac{1}{ \pi x^{n+1}} \left[ \int t^n \psi(t) \ dt + \lim_{\eps \rightarrow 0} \int_{|t|-x > \eps} \frac{t^{n+1} \psi(t)}{x-t} \ dt \right] \\
& = \frac{1}{x^{n+1}} \left[ \frac{1}{\pi} \int t^n \psi(t) \ dt + \Hil g(x) \right]
\end{align*}
where $g(x)=x^{n+1} \psi(x)$. Form Theorem \ref{decay} and the assumptions on $\psi(x)$, it follows that
\begin{equation*}
|x^{n+1} \Hil \psi(x)| \leq \frac{1}{\pi} \Big (\|x^n \psi(\cdot)\|_1 + 2\|g'\|_{\infty} +\|g\|_1 \Big).
\end{equation*}
Combining this with \eqref{bound5}, we obtain \eqref{bound4}.

	 As for the vanishing moments of $\Hil \psi(x)$, note that, since $\psi(x)$ has $n$ vanishing moments, 
\begin{equation*}
\int \lvert x^k \psi(x) \rvert \ dx < \infty \qquad (0 \leq k <n).
\end{equation*}
One can then verify, e.g., using the dominated convergence theorem, that $\hat{\psi}(\w)$ is $n$-times differentiable, and that
\begin{equation}
\label{equivalence}
\hat{\psi}^{(k)}(0)=(-j)^n  \int x^n \psi(x) \ dx \qquad (0 \leq k <n).
\end{equation}
Therefore,  $\hat{\psi}^{(k)}(0)=0$ for $0 \leq k < n$.

Now, since $\psi(x)$ is square-integrable\footnote{This follows from the fact that $f(x)$ is both integrable and bounded. The boundedness of $f(x)$ is a consequence of its uniform continuity, which in turn follows from the boundedness of $f'(x)$. Indeed, uniform continuity along with integrability implies that $f(x) \longrightarrow 0$ as $x \longrightarrow \infty$, and this along with continuity implies boundedness.}, \eqref{freq-def} holds. It  can then be verified that $\widehat{\Hil \psi}(\w)$ is $n$-times differentiable, and that $\widehat{\Hil \psi}^{(k)}(0)=0$ for $0 \leq k < n$. To arrive at the desired conclusion, we note that $|\Hil \psi(x)| \leq C/(1+|x|^{n+1})$, whereby
\begin{equation*}
\int \lvert x^k \Hil \psi(x) \rvert \ dx < \infty \qquad (0 \leq k <n).
\end{equation*}
This is sufficient to ensure that \eqref{equivalence} holds for $\Hil\psi(x)$, thus completing the proof.
\end{proof}

Note that the specialized form of this result is well-know for the particular case of $n=1$, that is, when the function is of zero mean. For example, along with the classical Cald\'eron-Zygmund decomposition (a wavelet-like decomposition), this is used to derive certain boundedness properties of the transform on the class of integrable functions \cite{Javier}. To the best of our knowledge, there is no explicit higher-order generalization of this result in the form of Theorem \ref{better-decay} in the harmonic analysis or signal processing literature.

\subsection{Smoothness}
		
	 We now investigate the smoothness of $\Hil \psi(x)$. The route we take capitalizes on the Fourier-domain specification of the transform, and the fact that the smoothness of a function is related to the decay of its Fourier transform. In general, the better the decay of the Fourier transform, the smoother is the function, and vice versa. We recall that a finite-energy signal $f(x)$ is said to belong to the Sobolev space $\mathbf{W}^{2,\gamma}(\R), \gamma \geq 0,$ if
\begin{equation*}
\int (1+|\w|^2)^\gamma |\hat f(\w)| ^2 d\w < \infty
\end{equation*}
The Sobolev embedding theorem asserts that every $f(x)$ belonging to $\mathbf{W}^{2,\gamma}(\R)$ can be identified (almost everywhere) with a function which is $n$-times continuously differentiable provided that $\gamma>n+1/2$; e.g., see \cite{WTSP}.  Since \eqref{freq-def} holds true for all finite-energy signals, we immediately conclude that
 
\begin{proposition}[Comparable smoothness]
\label{differentiability}
If $\psi(x)$ belongs to $\mathbf{W}^{2,\gamma}(\R)$, then $\Hil \psi(x)$ belongs $\mathbf{W}^{2,\gamma}(\R)$. In particular, if $\gamma>n+1/2$, then both $\psi(x)$ and $\Hil \psi(x)$ are $n$-times continuously differentiable (almost everywhere). 
\end{proposition}

	For example, the cubic spline wavelet belongs to $\mathbf{W}^{2,\gamma}(\R)$ for all $\gamma < 3+1/2$ \cite{FractionalSplines2000}. This explains the comparable smoothness of the wavelet and its transform shown in Figure \ref{HilbertWavelets}, which are both twice continuously differentiable. 

\section{Conclusion}

	It  has been known for quite some time that the Hilbert transform of a wavelet is again a wavelet. In this correspondence, we were concerned with the precise understanding of the sense in which this holds true. In particular, we formulated certain basic theorems concerning the localization, smoothness, and the number of vanishing moments of the Hilbert transform of a wavelet. Our main objective was to provide self-contained and straightforward proofs of these results along with some concrete examples.

%
%

     
\bibliographystyle{plain}
\bibliography{bibliography.bib}

\begin{thebibliography}{10}

\bibitem{Flandrin}
P.~Abry and P.~Flandrin.
\newblock Multiresolution transient detection.
\newblock {\em Proc. Int. Symp. on Time-Freq. and Time-Scale Analysis}, pages
  225--228, 1994.

\bibitem{Bedrosian}
E.~Bedrosian.
\newblock A product theorem for {H}ilbert transforms.
\newblock {\em Proc. {IEEE}}, 51:868--869, 1963.

\bibitem{Chaudhury2009}
K.~N. Chaudhury and M.~Unser.
\newblock Construction of {H}ilbert transform pairs of wavelet bases and
  {G}abor-like transforms.
\newblock {\em {IEEE} Trans. Signal Process.}, 57:3411--3425, 2009.

\bibitem{Pasquet}
C.~Chaux, L.~Duval, and J.-C. Pesquet.
\newblock Hilbert pairs of {M}-band orthonormal wavelet bases.
\newblock {\em Proc. Eur. Sig. and Image Proc. Conference}, pages 6--10, 2004.

\bibitem{Daubechis}
I.~Daubechies.
\newblock {\em Ten lectures on wavelets}.
\newblock Society for Industrial and Applied Mathematics, 1992.

\bibitem{Javier}
J.~Duoandikoetxea.
\newblock {\em Fourier {A}nalysis}.
\newblock American Mathematical Society, 2000.

\bibitem{kingsbury2}
N.~G. Kingsbury.
\newblock Complex wavelets for shift invariant analysis and filtering of
  signals.
\newblock {\em Journal of Applied and Computational Harmonic Analysis},
  10(3):234--253, May 2001.

\bibitem{kingsbury1}
N.G. Kingsbury.
\newblock Shift invariant properties of the dual-tree complex wavelet
  transform.
\newblock {\em Proc. {IEEE} International Conference on Acoustics, Speech, and
  Signal Processing ({ICASSP'99})}, March 1999.

\bibitem{WTSP}
S.~Mallat.
\newblock {\em A {W}avelet {T}our of {S}ignal {P}rocessing}.
\newblock San Diego, CA: Academic Press, 1998.

\bibitem{Olson}
T.~Olson and J.~DeStefano.
\newblock Wavelet localization of the {R}adon transform.
\newblock {\em {IEEE} Trans. Signal Process.}, 42:2055--2067, 1994.

\bibitem{CTDWT}
I.~W. Selesnick, R.~G. Baraniuk, and N.~G. Kingsbury.
\newblock The dual-tree complex wavelet transform.
\newblock {\em {IEEE} Sig. Proc. Magazine}, 22(6):123--151, November 2005.

\bibitem{Stein}
E.~M. Stein.
\newblock {\em Singular {I}ntegrals and {D}ifferentiability {P}roperty of
  {F}unctions}.
\newblock Princeton University Press, 1970.

\bibitem{convergence}
M.~Unser, A.~Aldroubi, and M.~Eden.
\newblock On the asymptotic convergence of {B}-spline wavelets to {G}abor
  functions.
\newblock {\em {IEEE} Trans. Inf. Theory}, 38(2):864--872, March 1992.

\bibitem{FractionalSplines2000}
M.~Unser and T.~Blu.
\newblock Fractional splines and wavelets.
\newblock {\em {SIAM} Review}, 42(1):43--67, March 2000.

\end{thebibliography}

\end{document}